\documentclass{article}
\usepackage[all]{xy}
\usepackage{latexsym,amssymb,amsfonts,amsmath,amsthm}
\usepackage{graphicx}
\usepackage{wrapfig}
\usepackage{stmaryrd}
\usepackage{mathtools}

  \newtheorem{definition}{Definizione}[section]
  \newtheorem{theorem}[definition]{Theorem}
  \newtheorem{lemma}[definition]{Lemma}
  \newtheorem{cor}[definition]{Corollary}
  \newtheorem{prop}[definition]{Proposition}
  
  \newtheorem{example}[definition]{Example}

 \CompileMatrices

\title{Extending the extensional level of the Minimalist Foundation to axiomatic set theories}
\author{Samuele Maschio, Pietro Sabelli}
\date{}

\begin{document}
\maketitle
\begin{abstract}
We introduce extensions by rules of the extensional level of the Minimalist Foundation which turn out to be equivalent to constructive and classical axiomatic set theories.
\end{abstract}
\section{Introduction}
Classical mathematics leans on a standard foundational theory, that is Zermelo-Fraenkel axiomatic set theory $\mathbf{ZF}$. The situation in costructive mathematics is very different: there are many foundational theories in the literature and no one of them has already reached the privileged status of ``standard''.

Moreover, the foundational tendency in constructive mathematics changed after Bishop's work (see \emph{A constructive Manifesto} in \cite{BB85}). The modern view on constructivism is far from that of Brouwer's intuitionism or that of Russian computable mathematics. The notion of \emph{compatibility} plays an important role nowaday: constructive mathematics is in fact understood by most mathematicians working in the field as ordinary mathematics done with intuitionistic logic and for this reason it must lay in a common core between classical mathematics, Brouwer's intuitionism and Russian computable mathematics. In particular, a foundational theory corresponding to such a notion of constructivism should be itself a common core between the main classical and intuitionistic, predicative and impredicative foundational theories available in the literature. Maietti and Sambin in \cite{mtt} identified the properties that such a common core foundation should satisfy for meet this requirement. Later in \cite{m09} Maietti proposed a precise foundational theory, called \emph{Minimalist Foundation} (for short $\mathbf{MF}$), satisfying these properties. The formal system $\mathbf{MF}$ consists of two levels formulated as dependent type theories: the intensional level $\mathbf{mTT}$ and the extensional level $\mathbf{emTT}$ connected by a setoid model of the second in the first. The intensional level should be an account of all the computational aspects of the theory, while the extensional level is the one in which ordinary mathematics should be performed. In particular, $\mathbf{mTT}$ should be compatible (as in fact it is shown in \cite{m09}) with type theoretic foundations like Martin-L\"of type theory \cite{NPS90} and Coquand's Calculus of Constructions \cite{COC}, while $\mathbf{emTT}$ should be compatible with axiomatic set-theoretical foundations. In \cite{m09} an argument for the compatibility of $\mathbf{emTT}$ with Aczel's constructive Zermelo-Fraenkel set theory $\mathbf{CZF}$ in \cite{czfnotes} is sketched. In this paper we want to make precise that statement by showing that $\mathbf{emTT}$ is compatible with axiomatic set theories $\mathbf{CZF}$, $\mathbf{IZF}$ and $\mathbf{ZF}$ in a stronger sense. We will in fact extend $\mathbf{emTT}$ with some rules, obtaining a type theory $\mathbf{emTT}_{\mathcal{T}}$ equivalent to the corresponding set theory $\mathcal{T}$. Each of these theories $\mathbf{emTT}_{\mathcal{T}}$ can be seen as an envelope of the theory $\mathcal{T}$ embodying (meta)theoretical concepts of set theory like those of definable class, definable set and $\Delta_0$-formula. 

\section{$\mathbf{emTT}$ in brief}
We present here briefly the extensional level $\mathbf{emTT}$ of the Minimalist Foundation (for a complete description see \cite{m09}) which is formulated as a type theory that can be seen as a variant of Martin-L\"of type theory in \cite{ML84}.

First, $\mathbf{emTT}$ contains four kinds of types which allow us to keep a distinction between logical and mathematical entities and different degrees of complexity:  
\begin{enumerate}
\item {\bf small propositions} include the falsum constant $\bot$ and propositional identities of terms in sets, and are closed under connectives and quantifiers with respect to sets;
\item {\bf propositions} include all small propositions and propositional identities, and are closed under connectives and quantifiers; 
    \item {\bf sets} include the empty set $\mathsf{N}_0$, a singleton set $\mathsf{N}_1$, all small propositions and are closed under constructors $\Sigma$, $\Pi$, $+$, $\mathsf{List}$ and under quotients of sets with respect to small equivalence relations.
    \item each set and each proposition is a collection and {\bf collections} are closed under $\Sigma$ and include power-collections of sets.
\end{enumerate}

Some keypoints are the following:
\begin{enumerate}
\item the elimination rules of propositional constructors act only toward propositions; for this reason the axiom of choice is not a theorem of $\mathbf{emTT}$ as it is in Martin-L\"of type theory, since in particular one cannot produce witnesses for existential statements in general;
\item propositions are proof-irrelevant; every term of a proposition is equal to a canonical term $\mathsf{true}$;
\item propositional identities reflect definitional equalities;
\item extensionality of functions holds:
$$\mathbf{emTT}\vdash(\forall f\in (\Pi x\in A)B)(\forall g\in (\Pi x\in A)B)$$
$$\qquad\qquad\qquad\qquad\qquad((\forall x\in A)(\mathsf{Ap}(f,a)=\mathsf{Ap}(g,a))\rightarrow f=g)$$
\end{enumerate}

\paragraph{Notational convention} For sake of readability we modify the syntax of $\mathbf{emTT}$ in \cite{m09} by writing $a=_{A}b$ instead of $\mathsf{Eq}(A,a,b)$.

\section{A presentation of axiomatic set theories}
In this paper, we will consider variants of axiomatic set theories where the language is a bit different from the usual one (in the most common approaches variables are the only terms), but which are more suitable for our purposes and equivalent to the traditional one. We will still call these theories $\mathbf{CZF}$, $\mathbf{IZF}$ and $\mathbf{ZF}$. The language of $\mathbf{IZF}$ and $\mathbf{ZF}$ consists of terms and formulas, including subclasses of $\Delta_0$-terms and $\Delta_0$-formulas, generated by the following clauses:
\begin{enumerate}
\item every variable $x$ is a term;
    \item $\emptyset$ and $\omega$ are terms;
    \item if $a$ and $b$ are terms, then $\{a,b\}$ is a term;
    \item if $a$ is a term, then $\bigcup a$ and $\mathcal{P}(a)$ are terms;
    \item if $x$ is a variable, $a$ is a term in which $x$ is not free\footnote{ This prevents us from terms like $\{x\in x|\,x=x\}$ which would lead to inconsistency.}, and $\varphi$ is a formula, then $\{x\in a|\,\varphi\}$ is a term;\footnote{In a term of the form $\{x\in a|\,\varphi\}$ the variable $x$ is bounded. This is the only term constructor which bounds variables.}
    \item every variable $x$ is a $\Delta_0$-term;
    \item $\emptyset$ and $\omega$ are $\Delta_0$-terms;
    \item if $a$ and $b$ are $\Delta_0$-terms, then $\{a,b\}$ is a $\Delta_0$-term;
    \item if $a$ is a $\Delta_0$-term, then $\bigcup a$ and $\mathcal{P}(a)$ are $\Delta_0$-terms;
    \item if $x$ is a variable, $a$ is a $\Delta_0$-term in which $x$ is not free, and $\varphi$ is a $\Delta_0$-formula, then $\{x\in a|\,\varphi\}$ is a $\Delta_0$-term;
    \item $\bot$ is a formula;
    \item if $a$ and $b$ are terms, then $a=b$ and $a\in b$ are formulas;
    \item if $\varphi$ and $\psi$ are formulas, then $\varphi\wedge \psi$, $\varphi\vee \psi$ and $\varphi\rightarrow \psi$ are formulas;
    \item if $\varphi$ is a formula and $x$ is a variable, then $\forall x\,\varphi$ and $\exists x\,\varphi$ are formulas.
        \item $\bot$ is a $\Delta_0$-formula;
    \item if $a$ and $b$ are $\Delta_0$-terms, then $a=b$ and $a\in b$ are $\Delta_0$-formulas;
    \item if $\varphi$ and $\psi$ are $\Delta_0$-formulas, then $\varphi\wedge \psi$, $\varphi\vee \psi$ and $\varphi\rightarrow \psi$ are $\Delta_0$-formulas;
    \item if $\varphi$ is a $\Delta_0$-formula, $x$ is a variable and $a$ is a $\Delta_0$-term in which $x$ is not free, then $\forall x\,(x\in a\rightarrow\varphi)$ and $\exists x\,(x\in a\wedge\varphi)$ are $\Delta_0$-formulas.
\end{enumerate}
The language of $\mathbf{CZF}$ consists of terms and formulas generated by the previous clauses with the following changes:
\begin{enumerate}
    \item $\mathcal{P}(a)$ is removed from clause 4.;
    \item clause 5. is modified by adding the condition that $\varphi$ is a $\Delta_0$-formula.
\end{enumerate}

From now on we will restrict ourselves, without loss of generality, only to formulas in which no variable appears both free and under the scope of a quantifier.

\noindent We will adopt some standard abbreviations: $\neg \varphi$ means $\varphi\rightarrow \bot$, $\top$ means $\bot \rightarrow \bot$, $\varphi\leftrightarrow \psi$ means $(\varphi\rightarrow \psi)\wedge(\psi\rightarrow \varphi)$, $a\subseteq b$ is a shorthand for the formula $\forall x(x\in a\rightarrow x\in b)$ where $x$ is a fresh variable and $\exists! x \, \varphi$ is an abbreviation for $\exists x \, \varphi \, \wedge \, \forall x \forall y (\varphi \, \wedge \, \varphi[y/x] \to x = y)$ where $y$ is a fresh variable; we also introduce bounded quantifiers with their usual meanings, namely $\exists x\in t\,\varphi$ means $\exists x(x\in t\wedge \varphi)$ and $\forall x\in t\,\varphi$ means $\forall x(x\in t \rightarrow \varphi)$; moreover, $0$ means $\emptyset$, $1$ means $\{\emptyset\}$, $\{a\}$ means $\{a,a\}$, $(a,b)$ means $\{\{a\},\{a,b\}\}$ and $a\cup b$ means $\bigcup\{a,b\}$. Finally
\begin{enumerate}
    \item $p_1(a)$ means $\bigcup\{x\in \bigcup a|\,\forall y(y\in a\rightarrow x\in y)\}$;
    \item $p_2(a)$ means $\bigcup\{x\in \bigcup a|\,x=p_1(a)\rightarrow a=\{\{p_1(a)\}\}\}$; 
    \item $\ell(a)$ means $\{x\in \omega|\,\exists y\,((x,y)\in a)\}$.
\end{enumerate}
The terms $p_1(a)$ and $p_2(a)$ represent the first and second component of $a$, respectively,  when $a$ has the form  $(b,c)$, while $\ell(a)$ represents the length of $a$, when $a$ is a list, that is a function whose domain is a natural number.

Besides the axioms and rules of intuitionistic first-order logic, the specific axioms of $\mathbf{IZF}$ are the universal closures of the following formulas:
\begin{enumerate}
    \item $\forall z(z\in x\leftrightarrow z\in y)\rightarrow x=y$
    \item $\neg (x\in \emptyset)$
    \item $x\in \{y,z\}\leftrightarrow x=y\vee x=z$
    \item $x\in \bigcup y\leftrightarrow \exists z(x\in z\wedge z\in y)$
    \item $x\in \mathcal{P}(y)\leftrightarrow x\subseteq y$
    \item $z\in \{x\in y|\,\varphi\}\leftrightarrow z\in y \wedge \varphi[z/x]$ for every formula $\varphi$;
    \item $0\in \omega \wedge \forall x \in \omega (x\cup\{x\}\in \omega)\wedge \forall y(0\in y\wedge \forall z\in y( z\cup \{z\}\in y)\rightarrow \omega\subseteq y) $
    \item $\forall x\in z\, \exists y\,\varphi\rightarrow \exists w\,\forall x\in z\,\exists y\in w\, \varphi$ for every formula $\varphi$ in which $w$ is not free;
    \item $\forall x(\forall y \in x \, \varphi[y/x] \rightarrow \varphi)\rightarrow \forall x\,\varphi$ for every formula $\varphi$ in which $y$ is not free.
\end{enumerate}
The axioms of $\mathbf{ZF}$ are those of $\mathbf{IZF}$ with the addition of the law of excluded middle, that is the universal closure of $\varphi\vee \neg\varphi$ for every formula $\varphi$.

The axioms of $\mathbf{CZF}$ are 1,2,3,4,7,9 in the list above and 6.\ with the obvious clause that $\varphi$ is a $\Delta_0$-formula, together with the universal closures of the following:
\begin{enumerate}
   \item[a.] $(\forall x\in z \exists y \varphi)\rightarrow \exists w(\forall x\in z\exists y\in w\,\varphi\wedge \forall y\in w\exists x\in z\,\varphi )$ for every formula $\varphi$ in which $w$ is not free.
    \item[b.] $\forall v\forall w\exists z\forall u(\forall x\in v\exists y\in w\,\varphi\rightarrow \exists z'\in z(\forall x\in v\exists y\in z'\,\varphi\wedge \forall y\in z'\exists x\in v\,\varphi))$ for every formula $\varphi$ in which $z$ is not free.
\end{enumerate}
These last two axiom schemas are called \emph{strong collection} and \emph{subset collection}, respectively. For further details on $\mathbf{CZF}$ the reader can refer to \cite{czfnotes}.

\section{The type theory $\mathbf{emTT_{\mathcal{T}}}$}\label{setrules}
For $\mathcal{T}$ being $\mathbf{CZF}$, $\mathbf{IZF}$ or $\mathbf{ZF}$, we define the type theory $\mathbf{emTT_{\mathcal{T}}}$ by adding rules to $\mathbf{emTT}$. 

The idea behind the extensions is that collections can be thought as definable classes of set theory, sets (of type theory) as definable classes which can be proven to be sets, while sets (of set theory) correspond to elements of a universal collection $\mathbf{V}$.

We will add the rules in the following four steps.

\subsection*{Step 1: Collections-as-definable-classes}
The first step consists in forcing the identification between collections of type theory and definable classes of set theory. These rules are included in $\mathbf{emTT}_{\mathcal{T}}$ for each $\mathcal{T}$.

The first thing to do is to introduce a universal collection $\mathbf{V}$\footnote{The universal collection $\mathbf{V}$ should not be intended as a universe in type-theoretical sense, but it has to be thought as a set-theoretical universe.}:
$$\cfrac{}{\mathbf{V}\,col}\qquad\cfrac{a\in A}{a\in \mathbf{V}}$$
We also need to norm the relation between definitional equality in an arbitrary collection and in the universal collection:
$$\cfrac{a=b\in A}{a=b\in \mathbf{V}}\qquad\cfrac{a\in A\qquad b\in A\qquad a=b\in \mathbf{V}}{a=b\in A}$$ $$ \cfrac{a\in A\qquad b\in \mathbf{V}\qquad a=b\in \mathbf{V}}{b\in A}$$
Moreover, we require the definitional equality with respect to $\mathbf{V}$ to be a small proposition:
$$\cfrac{a\in \mathbf{V}\qquad b\in \mathbf{V}}{a=_{\mathbf{V}}b\, prop_s}$$
We also need to introduce a new atomic small proposition (representing set-theoretic membership) and a new atomic proposition (internalizing type-theoretic membership):
$$ \cfrac{a\in \mathbf{V}\qquad b\in \mathbf{V}}{a\,\varepsilon\,b\,prop_s}\qquad
\cfrac{a\in \mathbf{V}\qquad A\,col}{a\,\varepsilon\,A\,prop}
$$
$$\cfrac{a\in A}{\mathsf{true}\in a\,\varepsilon\,A}\qquad \cfrac{A \; col \qquad \mathsf{true}\in a\,\varepsilon\, A}{a\in A}$$
with the relative rules saying that the new constructors are well-behaved with respect to definitional equality:
$$
\quad \cfrac{a = a'\in \mathbf{V}\qquad b = b'\in \mathbf{V}}{a\,\varepsilon\,b=a'\,\varepsilon\,b'\,prop_s}
\qquad
\cfrac{a=a'\in \mathbf{V}\qquad A=A'\,col}{a\,\varepsilon\,A=a'\,\varepsilon\,A'\,prop}
$$

We introduce now a new constructor which allows us to form collections as the result of a comprehension, that is, to include definable classes among collections
$$\cfrac{\varphi \,prop\,[x\in \mathbf{V}]}{\{x|\,\varphi\}\,col}$$
together with a rule describing the desired relationship between the new constructor and propositional membership, and with a rule of extensional equality for collections:\footnote{Note that, thanks to extensionality, we do not need to add a specific rule saying that the constructor $\{x|\,\varphi\}$ is well-behaved with respect to definitional equality. Indeed, the rule $$\cfrac{\varphi=\psi\,prop\,[x\in \mathbf{V}]}{\{x|\,\varphi\}=\{x|\,\psi\}\,col}$$
is derivable.}
$$\cfrac{\varphi\,prop\,[x\in \mathbf{V}]\qquad a\in \mathbf{V}}{\mathsf{true}\in \varphi[a/x]\leftrightarrow a\,\varepsilon\,\{x|\,\varphi\}}
\qquad
\cfrac{\mathsf{true}\in (\forall x\in \mathbf{V})(x\,\varepsilon\,A\leftrightarrow x\,\varepsilon\,B)}{A=B\,col}$$
The previous rules will, in turn, force the desired identification between collections and definable classes. Indeed, we can now derive the following rule:
$$
\cfrac{A\,col}{A=\{x|\,x\,\varepsilon\,A\}\,col}
$$

Finally, we add the following four rules aiming to describe bounded quantifiers in terms of quantifiers over the universal collection $\mathbf{V}$.

$$\cfrac{\varphi\,prop\,[x\in A]}{x\,\varepsilon\,A\wedge \varphi\,prop\,[x\in \mathbf{V}]}\qquad \cfrac{\varphi\,prop\,[x\in A]}{\mathsf{true}\in(\exists x\in A)\varphi\leftrightarrow(\exists x\in \mathbf{V})(x\,\varepsilon\,A\wedge \varphi)}$$

$$\cfrac{\varphi\,prop\,[x\in A]}{x\,\varepsilon\,A\rightarrow \varphi\,prop\,[x\in \mathbf{V}]}\qquad \cfrac{\varphi\,prop\,[x\in A]}{\mathsf{true}\in(\forall x\in A)\varphi\leftrightarrow(\forall x\in \mathbf{V})(x\,\varepsilon\,A\rightarrow \varphi)}$$

The two rules in the left side could look dangerous, since one could prove in $\mathbf{emTT}_{\mathcal{T}}$ that $b\varepsilon A\rightarrow \varphi[b/x]\,prop$ whenever $\varphi\,prop\,[x\in A]$ and $b\in \mathbf{V}$, without being able to prove  $\varphi[b/x]$ to be itself a proposition. However, this is not a real problem, since one can never prove that $\varphi[b/x]$ is true applying the elimination rules of conjunction and implication without being able to prove that $b$ is in $A$. 

Notice that in the practice of mathematics such expressions are nothing new, consider e.g.\ a proposition like $x\in \mathbb{N}^+\rightarrow \frac{x}{x}=1$, where the consequent $\frac{x}{x}=1$ makes sense only if we already know that $x$ is a positive number.

\subsection*{Step 2: Sets-as-definable-sets}
The rules in this step aim to identify type-theoretic sets with definable sets of set theory.
As in the previous step, they are all included in $\mathbf{emTT}_{\mathcal{T}}$ for every $\mathcal{T}$. 

The first rule states that if a collection is extensionally equal to an element of $\mathbf{V}$, then it is a set.  
$$\cfrac{A\, col\qquad \mathsf{true}\in (\exists y\in \mathbf{V}) (\forall x\in \mathbf{V})(x\,\varepsilon\, A\leftrightarrow x\,\varepsilon\, y)}{A\,set}$$
The converse is obtained by introducing a name in the universal collection $\mathbf{V}$ for each definable set: 
$$\cfrac{A\,set}{\lceil A\rceil\in \mathbf{V}}\qquad\qquad  \cfrac{A\,set}{\mathsf{true}\in (\forall x\in \mathbf{V})(x\,\varepsilon\,A\leftrightarrow x\,\varepsilon\,\lceil A\rceil)}$$

Finally, we add three rules which collapse all type definitional equalities to definitional equalities between collections.
$$\cfrac{A=B\,col\qquad A\,set\qquad B\,set}{A=B\,set}\qquad\cfrac{A=B\,col\qquad A\,prop\qquad B\,prop}{A=B\,prop}$$
$$\cfrac{A=B\,col\qquad A\,prop_s\qquad B\,prop_s}{A=B\,prop_s}$$

Notice that, using names $\lceil A\rceil$ in the universal collection $\mathbf{V}$, it can be easily shown that if two collections are equal and one of them is a set, then the other is a set too. 

\subsection*{Step 3: Axiomatic set theory via the universal collection}
In this step the axioms of set theory are embodied in the system via the universal collection $\mathbf{V}$.  We adopt here the usual abbreviations $\neg$ and $\leftrightarrow$.

The following group of rules is shared by $\mathbf{emTT}_{\mathbf{CZF}}$, $\mathbf{emTT}_{\mathbf{IZF}}$ and $\mathbf{emTT}_{\mathbf{ZF}}$:


$$\cfrac{a\in \mathbf{V}\qquad b\in \mathbf{V}}{\mathsf{true}\in (\forall x\in \mathbf{V})(x\,\varepsilon\,a\leftrightarrow x\,\varepsilon\,b)\rightarrow a=_{\mathbf{V}}b}$$

$$\cfrac{}{\emptyset\in \mathbf{V}}\qquad \cfrac{a\in \mathbf{V}}{\mathsf{true}\in \neg(a\,\varepsilon\,\emptyset)}$$

$$\cfrac{a\in \mathbf{V}\qquad b\in \mathbf{V}}{\{a,b\}\in \mathbf{V}}\qquad \cfrac{a\in \mathbf{V}\qquad b\in \mathbf{V}\qquad c\in \mathbf{V}}{\mathsf{true}\in c\,\varepsilon\,\{a,b\}\leftrightarrow c=_{\mathbf{V}} a \vee c=_{\mathbf{V}}b}$$

$$\cfrac{a\in \mathbf{V}}{\bigcup a\in \mathbf{V}}\qquad\cfrac{a\in \mathbf{V}\qquad b\in \mathbf{V}}{\mathsf{true}\in b\,\varepsilon\,\bigcup a \leftrightarrow (\exists x\in \mathbf{V})(b\,\varepsilon\,x\wedge x\,\varepsilon \,a)}$$

$$\cfrac{}{\omega \in \mathbf{V}}\qquad \cfrac{}{\mathsf{true}\in \mathsf{Trans}(\omega)}$$
$$\cfrac{a\in \mathbf{V}}{\mathsf{true}\in a\,\varepsilon\,\omega\rightarrow (\forall y\in \mathbf{V})(\mathsf{Trans}(y) \rightarrow a\,\varepsilon\, y)}$$
where $\mathsf{Trans}(y)$ is an abbreviation for 
$$\emptyset\,\varepsilon\,y\wedge (\forall z\in \mathbf{V})(z\,\varepsilon\,y\rightarrow \bigcup\{z,\{z,z\}\}\,\varepsilon\, y)$$


$$\cfrac{a\in \mathbf{V}\qquad \varphi\,prop\,[x\in \mathbf{V}]}{\mathsf{true}\in(\forall x\in \mathbf{V})( (\forall y\in \mathbf{V})(y\,\varepsilon\, x\rightarrow \varphi[y/x])\rightarrow \varphi)\rightarrow \varphi[a/x]}$$

\paragraph{Specific rules of $\mathbf{emTT}_{\mathbf{IZF}}$ and $\mathbf{emTT}_{\mathbf{ZF}}$.}
The following rules are included only in $\mathbf{emTT}_{\mathbf{IZF}}$ and in $\mathbf{emTT}_{\mathbf{ZF}}$:
$$\cfrac{a\in \mathbf{V}}{\mathcal{P}(a)\in \mathbf{V}}\qquad \cfrac{a\in \mathbf{V}\qquad b\in \mathbf{V}}{\mathsf{true}\in b\,\varepsilon \,\mathcal{P}(a)\leftrightarrow (\forall x\in \mathbf{V})(x\,\varepsilon\,b\rightarrow x\,\varepsilon\,a)}$$

$$\cfrac{a\in \mathbf{V}\qquad \varphi\,prop\,[x\in \mathbf{V}]}{\{x\,\varepsilon\, a|\,\varphi\}\in \mathbf{V}}\qquad \cfrac{a\in \mathbf{V}\qquad \varphi\,prop\,[x\in \mathbf{V}]\qquad b\in \mathbf{V}}{\mathsf{true}\in b\,\varepsilon\,\{x\,\varepsilon\,a|\,\varphi\}\leftrightarrow b\,\varepsilon\,a\wedge\varphi[b/x]}$$

$$\cfrac{\begin{array}{l}
a\in \mathbf{V}\qquad\varphi\,prop\,[x\in \mathbf{V},y\in \mathbf{V}]\\
\mathsf{true}\in (\forall x\in \mathbf{V})(x\,\varepsilon\,a\rightarrow (\exists y\in \mathbf{V})\varphi)\\
\end{array}}{\mathsf{true}\in(\exists z\in \mathbf{V})((\forall x\in \mathbf{V})(x\,\varepsilon\, a\rightarrow (\exists y\in \mathbf{V})(y\,\varepsilon\,z\wedge \varphi)))}$$

The law of excluded middle is obviously included only in $\mathbf{emTT}_{\mathbf{ZF}}$:
$$\cfrac{\varphi\,prop}{\mathsf{true}\in \varphi \vee \neg \varphi}$$

\paragraph{Specific rules of $\mathbf{emTT}_{\mathbf{CZF}}$.}
The following rules are included only in $\mathbf{emTT}_{\mathbf{CZF}}$:
$$\cfrac{a\in \mathbf{V}\qquad \varphi\,prop_s\,[x\in \mathbf{V}]}{\{x\,\varepsilon\, a|\,\varphi\}\in \mathbf{V}}\qquad \cfrac{a\in \mathbf{V}\qquad \varphi\,prop_s\,[x\in \mathbf{V}]\qquad b\in \mathbf{V}}{\mathsf{true}\in b\,\varepsilon\,\{x\,\varepsilon\,a|\,\varphi\}\leftrightarrow b\,\varepsilon\,a\wedge\varphi[b/x]}$$

$$\cfrac{\varphi\,prop\,[x\in \mathbf{V},y\in \mathbf{V},z\in \mathbf{V}]}{\mathsf{true}\in \mathsf{SCol}(\varphi)}$$
where $\mathsf{SCol}(\varphi)$ is: 
{\small $$(\forall x\in \mathbf{V})(x\,\varepsilon\, z\rightarrow (\exists y\in \mathbf{V}) \varphi)\rightarrow $$
$$(\exists w\in \mathbf{V})((\forall x\in \mathbf{V})(x\,\varepsilon\, z\rightarrow (\exists y\in \mathbf{V})(y\,\varepsilon\, w\wedge\varphi))\wedge (\forall y\in \mathbf{V})(y\,\varepsilon\, w \rightarrow(\exists x\in \mathbf{V})(x\,\varepsilon\, z\wedge\varphi )))$$}

$$\cfrac{\varphi\,prop\,[x\in \mathbf{V},y\in \mathbf{V},z\in \mathbf{V},v\in \mathbf{V},w\in \mathbf{V},u\in \mathbf{V},z'\in \mathbf{V}]}{\mathsf{true}\in \mathsf{SubCol}(\varphi)}$$
where $\mathsf{SubCol}(\varphi)$ is:
{\small $$(\forall v\in \mathbf{V})(\forall w\in \mathbf{V})(\exists z\in \mathbf{V})(\forall u \in \mathbf{V})\Big((\forall x\in \mathbf{V})(x\,\varepsilon\, v\rightarrow (\exists y\in \mathbf{V})(y\,\varepsilon\, w\wedge \varphi))\rightarrow $$
$$(\exists z'\in \mathbf{V})(z'\,\varepsilon\, z\wedge (\forall x\in \mathbf{V})(x\,\varepsilon\, v\rightarrow (\exists y\in \mathbf{V})(y\,\varepsilon\, z'\wedge\varphi))\wedge (\forall y\in \mathbf{V})(y\,\varepsilon\, z'\rightarrow (\exists x\in \mathbf{V})(x\,\varepsilon\, v\,\wedge \varphi)))\Big)$$ }

Let us conclude this step with two remarks. First, we can notice that the axiom of extensionality of set theory which we embodied in the system as the first rule above in this step guarantee that the newly defined term constructors relative to $\mathbf{V}$ (including $\lceil A\rceil$) are well-behaved with respect to definitional equality. Indeed, each of these constructors appears in the theory equipped with a rule which describes exactly its elements.

Moreover, we can establish a binary correspondance (up to the respective notions of equality) between type-theoretic sets and terms of type $\mathbf{V}$:
$$\qquad \qquad\; A\longmapsto \lceil A\rceil$$
$$\{x|\,x\,\varepsilon\,a\}\longmapsfrom a$$
Indeed we can derive from the content of Step 2
$$
\cfrac{A\,set}{A=\{x|\,x\,\varepsilon\,\lceil A\rceil\}\,set}
$$
while, using the axiom of extensionality in this step, we can prove
$$
\cfrac{a \in \mathbf{V}}{a=\big\lceil \{x|\,x\,\varepsilon\,a\}\big\rceil \in \mathbf{V}}
$$





\subsection*{Step 4: Interpretation-as-rules}
To recover the usual interpretation of types as sets (or classes), it suffices to specify how the canonical elements of each type are interpreted. This is done via the following rules which are included in $\mathbf{emTT}_{\mathcal{T}}$ for each $\mathcal{T}$.

$$\cfrac{}{\star = \emptyset \in \mathsf{N_1}}$$
$$\cfrac{A\,col\qquad B\,col\,[x\in A]\qquad a\in A\qquad b\in B[a/x]}{\langle a,b\rangle=(a,b)\in (\Sigma x\in A)B}$$
where $(t,s)$ means $\{\{t\},\{t,s\}\}$ and $\{t\}$ means $\{t,t\}$.
$$\cfrac{A\,set\qquad B\,set\,[x\in A]\qquad b\in B \, [x \in A]}{\lambda x^A.b = \{ z \,\varepsilon\, \lceil(\Sigma x\in A)B\rceil \, | \, (\exists x \in A)(z =_{\mathbf{V}} (x,b)) \} \in (\Pi x\in A)B }$$
where $z$ is a fresh variable.
$$\cfrac{A\, set\qquad B\,set\qquad a\in A}{\mathsf{inl}(a)=(\emptyset,a)\in A+B}\qquad\cfrac{A\, set\qquad B\,set\qquad b\in B}{\mathsf{inr}(b)=(\{\emptyset\},b)\in A+B}$$
$$\cfrac{A\,set}{\epsilon=\emptyset\in \mathsf{List}(A)}\qquad \cfrac{A\,set\qquad a\in \mathsf{List}(A)\qquad b\in A}{\mathsf{cons}(a,b)=\bigcup \{a,\{(\ell(a),b)\}\}\in \mathsf{List}(A)}$$
where $\ell(a):=\mathsf{El}_{\mathsf{List}}(a,\emptyset, (x,y,z)\bigcup\{z,\{z\}\})$ is a term of type $\{x|\,x\,\varepsilon\, \omega\}$.

$$\cfrac{\begin{array}{l}A\,set\qquad R\,prop_s\,[x\in A, y\in A]\\
\mathsf{true}\in (\forall x\in A)R[x/y]\\
\mathsf{true}\in (\forall x\in A)(\forall y\in A)(R\leftrightarrow R[y/x,x/y])\\
\mathsf{true}\in (\forall x\in A)(\forall y\in A)(\forall z\in A)(R\wedge R[y/x,z/y]\rightarrow R[z/y])
\\ a\in A\end{array}}{[a]=\{x\,\varepsilon\, \lceil A\rceil |\,R[a/y]\}\in A/R}$$

$$\cfrac{\varphi\,prop_s}{[\varphi]=\{x\,\varepsilon\,\{\emptyset\}|\,x=_{\mathbf{V}}\emptyset\wedge \varphi\}\in \mathcal{P}(1)}$$
where $x$ is a fresh variable.
$$\cfrac{A\,set\qquad b\in \mathcal{P}(1)\,[x\in A]}{\lambda x^A.b = \lceil\{ z  | \, (\exists x \in A)(z =_{\mathbf{V}} (x ,b)) \} \rceil\in A \to \mathcal{P}(1) }$$
where $z$ is a fresh variable and the rule of strong-collection in $\mathbf{CZF}$, or replacement and separation in case of $\mathbf{IZF}/\mathbf{ZF}$, together with the fact that $A$ is assumed to be a set, guarantee that the right-hand side of the conclusion is well-defined. 

$$\cfrac{\varphi\,prop \qquad \mathsf{true} \in \varphi}{\mathsf{true} = \emptyset \in \varphi}$$

Then, thanks to extensional equality for collections and the elimination and $\eta$-conversion rules of $\mathbf{emTT}$, we can derive the following rules characterizing sets and collections as definable classes and in which the variable $z$ is always assumed to be a fresh variable.

$$\cfrac{}{\mathsf{N}_{0}=\{z|\,\bot\}\,col}\qquad\cfrac{}{\mathsf{N}_{1}=\{z|\,z=_{\mathbf{V}}\emptyset\}\,col}$$
$$\cfrac{A\,col\qquad B\,col\,[x\in A]}{(\Sigma x\in A)B=\{z|\,(\exists x\in A)(\exists y\in B)(z=_{\mathbf{V}}(x,y))\}\,col}$$
$$\cfrac{A\,set\qquad B\,set\,[x\in A]}{(\Pi x\in A)B=\{z|\,\mathsf{Rel}(z,A,B)\wedge \mathsf{Svl}(z)\wedge \mathsf{Tot}(z,A)\}\,col}$$
where
\begin{enumerate}
    \item $\mathsf{Rel}(z,A,B)$ is $(\forall w\in \mathbf{V})(w\,\varepsilon\,z\rightarrow (\exists x\in A)(\exists y\in B)(w=_{\mathbf{V}}(x,y)))$
    \item $\mathsf{Svl}(z)$ is $(\forall x\in \mathbf{V})(\forall y\in \mathbf{V})(\forall y'\in \mathbf{V})((x,y)\,\varepsilon\,z\wedge (x,y')\,\varepsilon\, z\rightarrow y=_\mathbf{V}y')$
    \item $\mathsf{Tot}(z,A)$ is $(\forall x\in A)(\exists y\in \mathbf{V})((x,y)\,\varepsilon\,z)$
\end{enumerate}
$$\cfrac{A\,set\qquad B\,set}{A+B=\{z|\,(\exists y\in A)(z=_{\mathbf{V}}(\emptyset,y))\vee (\exists y\in B)(z=_{\mathbf{V}}(\{\emptyset\},y))\}\,col}$$
$$\cfrac{A\,set}{\mathsf{List}(A)=\{z|\,(\exists n\in \mathbf{V})(n\,\varepsilon\,\omega \wedge \mathsf{Rel}(z,n,A)\wedge \mathsf{Svl}(z)\wedge \mathsf{Tot}(z,n) )\}\,col}$$
where 
\begin{enumerate}
    \item $\mathsf{Rel}(z,n,A)$ is $(\forall w\in \mathbf{V})(w\,\varepsilon\,z\rightarrow (\exists x\in \mathbf{V})(\exists y\in A)(w=_{\mathbf{V}}(x,y)\wedge x\,\varepsilon\,n))$
    \item $\mathsf{Tot}(z,n)$ is $(\forall x\in \mathbf{V})(x\,\varepsilon\,n\rightarrow(\exists y\in \mathbf{V})((x,y)\,\varepsilon\,z))$
\end{enumerate}

$$\cfrac{\begin{array}{l}A\,set\qquad R\,prop_s\,[x\in A, y\in A]\\
\mathsf{true}\in (\forall x\in A)R[x/y]\\
\mathsf{true}\in (\forall x\in A)(\forall y\in A)(R\leftrightarrow R[y/x,x/y])\\
\mathsf{true}\in (\forall x\in A)(\forall y\in A)(\forall z\in A)(R\wedge R[y/x,z/y]\rightarrow R[z/y]) \\
\end{array}}{A/R=\{z|\,(\exists x\in A)(\forall y\in \mathbf{V})(y\,\varepsilon\, z\leftrightarrow y\,\varepsilon\,A\wedge R)\}\,col}$$
$$\cfrac{}{\mathcal{P}(1)=\{z|\,(\forall y\in \mathbf{V})(y\,\varepsilon\,z\rightarrow y=_{\mathbf{V}}\emptyset)\}\,col}$$
$$\cfrac{A\,set}{A\rightarrow \mathcal{P}(1)=\{z|\,\,\mathsf{Rel}(z,A,\mathcal{P}(1))\wedge \mathsf{Svl}(z)\wedge \mathsf{Tot}(z,A)\}\,col}$$
$$\cfrac{\varphi\,prop}{\varphi=\{z|\,z=_{\mathbf{V}}\emptyset\wedge \varphi\}\,col}$$

As a byproduct of this last rule together with the fact that propositions which are equal as collections are equal (as propositions), we obtain that two propositions $\varphi$ and $\psi$ are equal if and only if they are equivalent, that is $\mathsf{true}\in \varphi\leftrightarrow \psi$.

Finally, notice that we do not need to add rules for the interpretation of the elimination terms, since the relative computation rules in $\mathbf{emTT}$ suffice to uniquely determine them.

\section{Translations}\label{sec}
In this section we introduce two translations: one from the syntax of the set theory $\mathcal{T}$ to the pre-syntax of $\mathbf{emTT}_{\mathcal{T}}$. The other one in the opposite direction.
\\
\\
The pre-syntax of $\mathbf{emTT}_{\mathcal{T}}$ (for $\mathcal{T}$ equal to $\mathbf{CZF}$, $\mathbf{IZF}$ and $\mathbf{ZF}$) is defined as the following grammar, where $A$ and $B$ are metavariables for pre-collections, $a,b$ and $c$ for pre-terms,  $\varphi$ and $\psi$ for pre-propositions, and $x,y$ and $z$ for variables.
\begin{enumerate}
\item[] $A\textnormal{ pre-collection } ::=$
\begin{enumerate}
\item[] $\,\mathsf{N}_0 |  \, \mathsf{N}_1  |  \, \mathsf{List}(A)  |  \, A + B |  \, (\Sigma x \in A)B  | \, (\Pi x \in A)B | $
\item[] $\, A/\varphi |  \, \mathcal{P}(1)  |  \, A \to \mathcal{P}(1)  |  \, \{ x |\, \varphi \} |  \, \varphi  |  \, \mathbf{V}$
\end{enumerate}
\item[] $a\textnormal{ pre-term}\footnote{Notice that we decided to annotate some of the pre-terms in order to keep track of pieces of information which are crucial for an effective translation.} ::=$
\begin{enumerate}
   \item[]   $ \, x 
    |  \, \mathsf{emp}_0(a) 
    |  \, \star 
    |  \, \mathsf{El}_{\mathsf{N_1}}(a,b) 
    |  \, \epsilon 
    |  \, \mathsf{cons}(a,b) 
    |  \, \mathsf{El}^A_{\mathsf{List}}(a,b,(x,y,z)c)| $
    \item[]
    $  \, \mathsf{inl}(a) 
    |  \, \mathsf{inr}(a) 
    |  \, \mathsf{El}_+(a,(x)b,(y)c) 
    |  \, \langle a, b \rangle 
    |  \, \mathsf{El}_\Sigma(a,(x,y)b) 
    |  \, \lambda x^A.a 
    |  \, \mathsf{Ap}(a,b)| $
    \item[]
    $ \, [a]_{A,(x,y)\varphi} 
    |  \, \mathsf{El}_{A/(x,y)\varphi}(a,(x)b) 
    |  \, \mathsf{true} 
    | \, [\varphi]
    |  \, \lceil A \rceil 
    |  \, \emptyset 
    |  \, \{ a, b \} 
    |  \, \bigcup a 
    |  \, \mathcal{P}(a) 
    |  \, \{ x \varepsilon a |\, \varphi \} 
    |  \, \omega$
\end{enumerate}
\item[] $\varphi\textnormal{ pre-proposition } ::=$
\begin{enumerate}
\item[]
    $   \, \bot 
    |  \, a \,\varepsilon\, b 
    |  \, a \,\varepsilon\, A 
    |  \, a=_Ab 
    |  \, \varphi \to \psi 
    |  \, \varphi \wedge \psi 
    |  \, \varphi \vee \psi 
    |  \, (\exists x \in A)\varphi 
    |  \, (\forall x \in A)\varphi $
\end{enumerate}

\end{enumerate}

\emph{Pre-contexts} of $\mathbf{emTT}_{\mathcal{T}}$ are finite lists of declarations of variables in pre-collection defined by the following clauses:
\begin{enumerate}
    \item the empty list $[\;]$ is a pre-context;
    \item if $\Gamma$ is a pre-context, $x$ is a variable not appearing in $\Gamma$ and $A$ is a pre-collection, then $[\Gamma, x\in A]$ is a pre-context.
\end{enumerate}

The first translation is then defined as follows:
\begin{definition} 
 Every term $a$ of $\mathcal{T}$ is translated into a pre-term $\widetilde{a}$ of $\mathbf{emTT}_{\mathcal{T}}$ and every formula $\varphi$ of $\mathcal{T}$ is translated into a pre-proposition $\widetilde{\varphi}$ of $\mathbf{emTT}_{\mathcal{T}}$ according to the following clauses:
\begin{enumerate}
  \item $\widetilde{x}:=x$;
     \item $\widetilde{\emptyset}:=\emptyset$ and $\widetilde{\omega}:=\omega$;
     \item $\widetilde{\{a,b\}}:=\{\widetilde{a},\widetilde{b}\}$, $\widetilde{\bigcup a}:=\bigcup \widetilde{a}$ and $\widetilde{\mathcal{P}(a)}:=\mathcal{P}(\widetilde{a})$;
     \item $\widetilde{\{x\in a|\,\varphi\}}:=\{x\,\varepsilon\, \widetilde{a}|\,\widetilde{\varphi}\}$;
    \item $\widetilde{a=b}:\equiv \widetilde{a}=_{\mathbf{V}}\widetilde{b}$ and $\widetilde{a\in b}:\equiv \widetilde{a}\,\varepsilon\,\widetilde{b}$;
    \item $\widetilde{\bot}:\equiv\bot$;
    \item $\widetilde{\varphi \wedge \psi}:\equiv\widetilde{\varphi}\,\kappa\, \widetilde{\psi}$ for $\kappa$ being $\wedge$, $\vee$ or $\rightarrow$;
    \item $\widetilde{\exists x\,\varphi}:\equiv(\exists x\in \mathbf{V})\widetilde{\varphi}$ and  $\widetilde{\forall x\,\varphi}:\equiv(\forall x\in \mathbf{V})\widetilde{\varphi}$
\end{enumerate}
\end{definition}

Now we define the second translation. The idea behind this translation is that: 
\begin{enumerate}
\item each pre-collection of $\mathbf{emTT}_{\mathcal{T}}$ is translated into a formula $\eta_A$ depending on a fresh variable $u$ which would determine the interpretation of $A$ by comprehension: $A$ should be interpreted as $\{u|\,\eta_A\}$;
    \item each pre-term $a$ of $\mathbf{emTT}_{\mathcal{T}}$ is translated into a formula $\delta_a$ of $\mathcal{T}$, depending on a fresh variable $u$, providing a well-defined interpretation of $a$: $a$ should be interpreted as the unique $u$ for which $\delta_a$ holds;
    \item each pre-propositions of $\mathbf{emTT}_{\mathcal{T}}$ is translated into a formula of $\mathcal{T}$ having the same free variables.

\end{enumerate}

\begin{definition} Let $u$ be a new variable added to the language of $\mathcal{T}$.\footnote{The choice of adding a new variable to the language of set theory is dictated by the advantage which the use of a uniform fresh variable could bring. In fact, the role of $u$ is nothing but that of a placeholder.} For every pre-collection $A$ of $\mathbf{emTT}_{\mathcal{T}}$ we define a formula $\eta_A$, for every pre-term $a$ of $\mathbf{emTT}_{\mathcal{T}}$ we define a formula $\delta_A$ of $\mathbf{emTT}_{\mathcal{T}}$ and we translate every pre-proposition $\varphi$ of $\mathbf{emTT}_{\mathcal{T}}$ into a formula $\widehat{\varphi}$ of $\mathcal{T}$ according to the following clauses, where the variables $v,w,w',w'',w_1,w_2,w_3$ and $n$ are meant to be fresh:
\begin{enumerate}
\item $\widehat{\bot}:\equiv\bot$;
\item $\widehat{a\,\varepsilon\,b}:\equiv \exists u\exists v(\delta_a\wedge \delta_b[v/u]\wedge u\in v)$; 
\item $\widehat{a\,\varepsilon\,A}:\equiv \exists u(\delta_a\wedge \eta_A)$;
\item $\widehat{a=_A b}:\equiv \exists u(\delta_a\wedge \delta_b\wedge \eta_A) $;
\item $\widehat{\varphi\, \kappa\, \psi}:\equiv\widehat{\varphi}\,\kappa\, \widehat{\psi}$ for $\kappa$ being $\wedge$, $\vee$ or $\rightarrow$;
\item $\widehat{(\forall x\in A)\varphi}:\equiv\forall x(\eta_A[x/u]\rightarrow \widehat{\varphi})$;
\item $\widehat{(\exists x\in A)\varphi}:\equiv\exists x(\eta_A[x/u]\wedge \widehat{\varphi})$;
\item $\delta_x:\equiv u=x$;
\item $\eta_{\varphi}:\equiv u=0\wedge \widehat{\varphi}$ and $\delta_{\mathsf{true}}:\equiv u=0$;
\item $\eta_{\mathsf{N}_0}:\equiv\bot$ and  $\delta_{\mathsf{emp}_0(a)}:\equiv u=0$;
\item $\eta_{\mathsf{N}_1}:\equiv u=0$, $\delta_{\star}:\equiv u=0$ and  $\delta_{\mathsf{El}_{\mathsf{N}_{1}}(a,b)}:\equiv\delta_b$;
\item $\eta_{(\Sigma x\in A)B}:\equiv \exists v\exists w(\eta_A[v/u]\wedge \eta_B[v/x,w/u]\wedge u=(v,w))$, 
\item[] $\delta_{\langle a,b\rangle}:\equiv\exists v \exists w(\delta_a[v/u]\wedge \delta_b[w/u]\wedge u=(v,w))$ and 
\item[] $\delta_{\mathsf{El}_{\Sigma}(a,(x,y)b)}:\equiv\exists v(\delta_{a}[v/u]\wedge \delta_{b}[p_1(v)/x,p_2(v)/y])$;
\item $\eta_{(\Pi x\in A)B}:\equiv \forall v(v\in u \rightarrow \exists w\exists w'(v=(w,w')\wedge \eta_A[w/u]\wedge \eta_B[w/x,w'/u])) \wedge$
\item[] $   \qquad\qquad\qquad\forall w\forall w'\forall w''((w,w')\in u \wedge (w,w'')\in u\rightarrow w'=w'')\wedge$
\item[] $\qquad\qquad\qquad\forall w(\eta_A[w/u]\rightarrow \exists w'((w,w')\in u))$,
\item[] $\delta_{\lambda x^A. b}:\equiv\forall v(v\in u \leftrightarrow \exists w\exists w'(\eta_A[w/u]\wedge \delta_b[w/x,w'/u]\wedge v=(w,w')))$ and
\item[] $\delta_{\mathsf{Ap}(a,b)}:\equiv\exists v\exists w(\delta_a[v/u]\wedge \delta_b[w/u]\wedge u=p_2(\bigcup\{z\in v|\,p_1(z) = w\}))$; 
\item $\eta_{A+B}:\equiv\exists v(\eta_A[v/u]\wedge u=(0,v))\vee \exists w(\eta_B[w/u]\wedge u=(1,w))$,
\item[] $\delta_{\mathsf{inl}(a)}:\equiv \exists v(\delta_{a}[v/u]\wedge u=(0,v))$, $\delta_{\mathsf{inr}(a)}:\equiv \exists v(\delta_{a}[v/u]\wedge u=(1,v))$ and
\item[] $\delta_{\mathsf{El}_{+}(a,(x)b,(y)c)}:\equiv$
\item[] $\qquad\qquad\exists v(\delta_a[v/u]\wedge ((p_1(v)=0\wedge \delta_b[p_2(v)/x])\vee(p_1(v)=1\wedge \delta_c[p_2(v)/y])))$;
\item $\eta_{\mathsf{List}(A)}:\equiv\exists n(n\in \omega\wedge$
\item[]$\qquad\qquad\qquad\qquad \forall v(v\in u \leftrightarrow \exists w\exists w'(w\in n\wedge \eta_A[w'/u]\wedge v=(w,w')))\wedge$
\item[]$\qquad \qquad\qquad\qquad \forall w\forall w'\forall w''((w,w')\in u\wedge (w,w'')\in u\rightarrow w'=w'')\wedge$
\item[] $\qquad \qquad\qquad\qquad \forall w(w\in n\rightarrow \exists w'((w,w')\in u)) )$,
\item[] $\delta_{\epsilon}:\equiv u=0$,
\item[] $\delta_{\mathsf{cons}(a,b)}:\equiv\exists v\exists w(\delta_a[v/u]\wedge \delta_b[w/u]\wedge u=v\cup\{(\ell(v),w)\})$ and 
\item[] $\delta_{\mathsf{El}^A_{\mathsf{List}}(a,b,(x,y,z)c)}:\equiv \exists f\Big($
\item[] $\forall w(w\in f\rightarrow \exists w_1\exists w_2(\eta_{\mathsf{List}(A)}[w_1/u]\wedge w=(w_1,w_2)) \wedge$
\item[] $\forall w_1\forall w_2\forall w_3((w_1,w_2)\in f\wedge (w_1,w_3)\in f\rightarrow w_2=w_3)\wedge$
\item[] $\forall w_1(\eta_{\mathsf{List}(A)}[w_1/u]\rightarrow \exists w_2((w_1,w_2)\in f))\wedge$
\item[] $\exists v(\delta_b[v/u]\wedge (0,v)\in f)\wedge$
\item[] $\forall w_1\forall w_2\forall w_3\forall v((w_1,w_3)\in f\wedge \eta_A[w_2/u]\wedge \delta_c[w_1/x,w_2/y,w_3/z,v/u]\rightarrow $
\item[] $(w_1\cup \{(\ell(w_1),w_2)\},v)\in f )\wedge$
\item[] $\exists w'(\delta_a[w'/u]\wedge (w',u)\in f)\Big)$;
\item $\eta_{A/(x,y)\varphi}:\equiv\exists w(\eta_A[w/u]\wedge \forall v(v\in u\leftrightarrow \eta_A[v/u]\wedge \widehat{\varphi}[w/x,v/y]))$,
\item[] $\delta_{[a]_{A,(x,y)\varphi}}:\equiv\exists w(\delta_a[w/u]\wedge\forall v(v\in u\leftrightarrow \eta_A[v/u]\wedge \widehat{\varphi}[w/x,v/y]))$ and 
\item[] $\delta_{\mathsf{El}_Q(a,(x)b)}:\equiv\exists v(\delta_a[v/u]\wedge \exists w(w\in v)\wedge \forall w(w\in v\rightarrow \delta_b[w/x]))$;
\item $\eta_{\mathcal{P}(1)}:\equiv u\subseteq \{0\}$ and $\delta_{[\varphi]}:\equiv\forall v(v\in u \leftrightarrow v=0\wedge \widehat{\varphi})$;
\item $\eta_{A\rightarrow \mathcal{P}(1)}:\equiv \eta_{(\Pi x\in A)\mathcal{P}(1)}$ where $x$ is fresh;
\item $\eta_{\mathbf{V}}:\equiv u=u$,
\item[] $\delta_{\lceil A\rceil}:\equiv\forall v(v\in u\leftrightarrow \eta_A[v/u])$,
\item[] $\delta_{\emptyset}:\equiv u=\emptyset$,
\item[] $\delta_{\{a,b\}}:\equiv\exists v\exists w(\delta_a[v/u]\wedge \delta_b[w/u]\wedge u=\{v,w\})$,
\item[] $\delta_{\bigcup a}:\equiv\exists v(\delta_a[v/u]\wedge u=\bigcup v)$,
\item[] $\delta_{\mathcal{P}(a)}:\equiv\exists v(\delta_a[v/u]\wedge \forall w(w\in u \leftrightarrow w\subseteq v))$,
\item[] $\delta_{\{x\,\varepsilon\,a|\,\varphi\}}:\equiv\exists v(\delta_{a}[v/u]\wedge \forall x(x\in u\leftrightarrow x\in v\wedge \widehat{\varphi}))$ and
\item[] $\delta_{\omega}:\equiv u=\omega$;
\item $\eta_{\{x|\,\varphi\}}:\equiv \widehat{\varphi}[u/x]$.
\end{enumerate}
If $\Gamma$ is a pre-context of $\mathbf{emTT}_{\mathcal{T}}$, we define the formula $\widehat{\Gamma}$ of $\mathcal{T}$ as follows:
\begin{enumerate}
    \item $\widehat{[\;]}:\equiv \top$
    \item $\widehat{[\Gamma,x\in A]}:\equiv \widehat{\Gamma}\wedge \eta_A[x/u]$
\end{enumerate}
\end{definition}


We now show that the composition of the two translations in one order results in an equivalence.

\begin{prop}\label{oneside}
Let $a$ be a term of $\mathcal{T}$ and $\psi$ a formula of $\mathcal{T}$. Then:
\begin{enumerate}
    \item $\mathcal{T}\vdash u=a\leftrightarrow \delta_{\widetilde{a}}$; 
\item $\mathcal{T}\vdash \psi\leftrightarrow \widehat{\widetilde{\psi}}$
\end{enumerate}
\end{prop}
\begin{proof}
1.\ and 2.\ are proven simultaneously by induction on the complexity of terms and formulas. Let us consider the only non-trivial cases: terms of the form $\{x\in a|\,\psi\}$ and atomic formulas $a=b$ and $a\in b$:
\begin{enumerate}
    \item $\delta_{\widetilde{\{x\in a|\,\psi\}}}$ is by definition $\exists v(\delta_{\widetilde{a}}[v/u]\wedge \forall x(x\in u\leftrightarrow x\in v\wedge\widehat{\widetilde{\psi}} ))$. Using the inductive hypothesis 1.\ on the term $a$ we have that this is equivalent in $\mathcal{T}$ to $\forall x(x\in u\leftrightarrow x\in a\wedge \widehat{\widetilde{\psi}} ))$; using inductive hypothesis 2.\ on the formula $\psi$ this is equivalent in $\mathcal{T}$ to $\forall x(x\in u\leftrightarrow x\in a\wedge \psi ))$ which is equivalent to $u=\{x\in a|\,\psi\}$;
    \item $\widehat{\widetilde{a=b}}$ is equivalent by definition to $\exists u(\delta_{\widetilde{a}}\wedge \delta_{\widetilde{b}})$ which is equivalent by inductive hypothesis 1.\ on $a$ and $b$ to $\exists u\,(u=a\,\wedge\, u=b)$ that is equivalent to $a=b$;
    \item the proof of the equivalence between $\widehat{\widetilde{a\in b}}$ and $a\in b$ in $\mathcal{T}$ is analogous to the previous one.
\end{enumerate}
\end{proof}
\noindent The next step consists in proving that formulas of the form $\delta_t$ hold at most for one $u$ in $\mathcal{T}$. 

\begin{lemma} Let $t$ be a pre-term of $\mathbf{emTT}_{\mathcal{T}}$. Then
$$\mathcal{T}\vdash \delta_{t}\wedge \delta_{t}[v/u]\rightarrow u=v$$
\end{lemma}
\begin{proof}
This is a straightforward proof by induction on complexity of $t$.
\end{proof}
\begin{lemma}[Substitution Lemma] Let $t$ and $a$ be pre-terms of $\mathbf{emTT}_{\mathcal{T}}$, $A$ a pre-collection of $\mathbf{emTT}_{\mathcal{T}}$ and $\varphi$ a pre-proposition of $\mathbf{emTT}_{\mathcal{T}}$. Then:
\begin{enumerate}

   
    \item $\mathcal{T}\vdash \exists u\delta_{t}\wedge \delta_{a[t/x]}\leftrightarrow \exists v(\delta_{t}[v/u]\wedge \delta_{a}[v/x])$
    \item $\mathcal{T}\vdash \exists u\delta_{t}\wedge \eta_{A[t/x]}\leftrightarrow \exists v(\delta_{t}[v/u]\wedge \eta_{A}[v/x])$
    \item $\mathcal{T}\vdash \exists u \delta_t\wedge \widehat{\varphi[t/x]}\leftrightarrow \exists v(\delta_{t}[v/u]\wedge \widehat{\varphi}[v/x])$
\end{enumerate}
where $v$ is always assumed to be a fresh variable.
\end{lemma}
\begin{proof}
These are proven simultaneously by induction on the complexity of the pre-syntax of $\mathbf{emTT}_{\mathcal{T}}$, once one notice that the statements are equivalent to the following, respectively:
\begin{enumerate}
    \item $\mathcal{T}\vdash\delta_t[v/u]\rightarrow (\delta_{a[t/x]}\leftrightarrow \delta_a[v/x])$
    \item $\mathcal{T}\vdash\delta_t[v/u]\rightarrow (\eta_{A[t/x]}\leftrightarrow \eta_A[v/x])$
    \item $\mathcal{T}\vdash\delta_t[v/u]\rightarrow (\widehat{\varphi[t/x]}\leftrightarrow \widehat{\varphi}[v/x])$
\end{enumerate}
\end{proof}

The next proposition is the counterpart of Proposition \ref{oneside}.
\begin{prop}\label{otherside}
Let $A$ be a pre-collection of $\mathbf{emTT}_{\mathcal{T}}$, $a$ a pre-term of $\mathbf{emTT}_{\mathcal{T}}$ and $\varphi$ a pre-proposition of $\mathbf{emTT}_{\mathcal{T}}$. Then:
\begin{enumerate}
    \item if $\mathbf{emTT}_{\mathcal{T}}\vdash A\,col\,[\Gamma]$, then $$\mathbf{emTT}_{\mathcal{T}}\vdash A=\{z|\,\widetilde{\eta_A[z/u]}\}\,col\,[\Gamma]$$
     \item if $\mathbf{emTT}_{\mathcal{T}}\vdash a\in A\,[\Gamma]$, then $$\mathbf{emTT}_{\mathcal{T}}\vdash \mathsf{true}\in(\forall z\in \mathbf{V})\,(\widetilde{\delta_a[z/u]}\leftrightarrow z=_{\mathbf{V}}a)\,[\Gamma]$$
     \item if $\mathbf{emTT}_{\mathcal{T}}\vdash \varphi\,prop\,[\Gamma]$, then $$\mathbf{emTT}_{\mathcal{T}}\vdash \mathsf{true}\in\varphi\leftrightarrow \widetilde{\widehat{\varphi}}\,[\Gamma]$$
     \end{enumerate}
     where $z$ is always a fresh variable.

\end{prop}

\section{The main result}
The first theorem says that $\mathbf{emTT}_{\mathcal{T}}$ can be seen as an extension of $\mathcal{T}$.
\begin{theorem}\label{emttzf}
Let $\psi$ be a formula of $\mathcal{T}$ whose free variables are among $x_1,...,x_n$. 
If $\mathcal{T}\vdash \psi$, then $$\mathbf{emTT}_{\mathcal{T}}\vdash \mathsf{true}\in\widetilde{\psi}\,[x_1\in \mathbf{V},...,x_n\in \mathbf{V}]$$
\end{theorem}
\begin{proof}
This is essentially an immediate consequence of the rules in Step 3 in Section \ref{setrules} and the rules for propositions in $\mathbf{emTT}$.
\end{proof}
\noindent The next theorem shows how the judgements of $\mathbf{emTT}_{\mathcal{T}}$ are interpreted in $\mathcal{T}$. However, before proceeding we need to introduce the concept of $\mathsf{K}_{0}$-formula relative to a formula $\gamma$ of $\mathcal{T}$. $\mathsf{K}_0$-formulas are nothing but $\Delta_0$ formulas in which some variables are substituted by definable elements which can possibly lack a representation as terms.
The class of formulas $\widetilde{\mathsf{K}_{0}}[\gamma]$ is the smallest one respecting the following clauses:
 \begin{enumerate}
     \item $\bot$, $x=y$ and $x\in y$ are in $\widetilde{\mathsf{K}_{0}}[\gamma]$ for every pair of variables $x,y$;
     \item if $\varphi$ and $\psi$ are in $\widetilde{\mathsf{K}_{0}}[\gamma]$, then $\varphi\wedge \psi$, $\varphi\vee \psi$ and $\varphi\rightarrow \psi$ are in $\widetilde{\mathsf{K}_{0}}[\gamma]$;
     \item if $\varphi$ is a formula in $\widetilde{\mathsf{K}_{0}}[\gamma]$, $y$ is a variable, $z$ is a fresh variable and $\delta$ is a formula such that $\mathcal{T}\vdash \gamma\rightarrow \exists ! z\,\delta$, then $\exists z(\delta\wedge \exists y\in z\,\varphi)$, $\exists z(\delta\wedge \forall y\in z\,\varphi)$ and $\exists z(\delta\wedge \varphi)$ are in $\widetilde{\mathsf{K}_{0}}[\gamma]$.
 \end{enumerate}
 The class $\mathsf{K}_{0}[\gamma]$ contains those formulas $\varphi$ in $\widetilde{\mathsf{K}_{0}}[\gamma]$ such that $\mathbf{free}(\varphi)\subseteq \mathbf{free}(\gamma)$.
 
Moreover, one can associate to every formula $\varphi$ in $\mathsf{K}_{0}[\gamma]$ a $\Delta_0$-formula $\sigma(\varphi)$ of $\mathcal{T}$ as follows:
  \begin{enumerate}
     \item $\sigma(\bot):\equiv\bot$, $\sigma(x=y):\equiv x=y$ and $\sigma(x\in y):\equiv x\in y$;
     \item if $\varphi$ and $\psi$ are in $\mathsf{K}_{0}[\gamma]$, then $\sigma(\varphi\wedge \psi):\equiv\sigma(\varphi)\wedge \sigma(\psi)$, $\varphi\vee \psi:\equiv\sigma(\varphi)\vee \sigma(\psi)$ and $\sigma(\varphi\rightarrow \psi):\equiv\sigma(\varphi)\rightarrow \sigma(\psi)$;
     \item if $\varphi$ is a formula in $\mathsf{K}_{0}[\gamma]$, $y$ is a variable and $z$ is a fresh variable and $\delta$ is a formula such that $\mathcal{T}\vdash \gamma\rightarrow \exists ! z\,\delta$, then $\sigma(\exists z(\delta\wedge \exists y\in z\,\varphi)):\equiv \exists y\in z\,\sigma(\varphi)$, $\sigma(\exists z(\delta\wedge \forall y\in z\,\varphi)):=\forall y\in z\,\sigma(\varphi)$ and $\sigma(\exists z(\delta\wedge \varphi)):\equiv \sigma(\varphi)$.
 \end{enumerate}
 Using this fact one can prove the following lemma using $\Delta_0$-separation which works in every $\mathcal{T}$.
 \begin{lemma}
 If $\varphi$ is a formula in $\mathsf{K}_{0}[\gamma]$, $v,v'$ are fresh variables and $x$ is a variable, then $$\mathcal{T}\vdash \gamma\rightarrow\forall v\exists v'\forall x(x\in v'\leftrightarrow x\in v\wedge \varphi).$$
 \end{lemma}
 
\begin{theorem}
The following hold:
\begin{enumerate}
    \item if $\mathbf{emTT}_{\mathcal{T}}\vdash A=B\,type\,[\Gamma]$ (for $type$ being $col$, $set$, $prop$ or $prop_s$), then $\mathcal{T}\vdash \widehat{\Gamma}\rightarrow\forall u(\eta_A\leftrightarrow \eta_B)$;
    \item if $\mathbf{emTT}_{\mathcal{T}}\vdash A\,set\,[\Gamma]$, then $\mathcal{T}\vdash \widehat{\Gamma}\rightarrow\exists z\forall u(u\in z\leftrightarrow \eta_A)$ (where $z$ is a free variable);
    \item if $\mathbf{emTT}_{\mathcal{T}}\vdash \varphi\,prop_s\,[\Gamma]$, then there exists a formula $\psi$ in $\mathsf{K}_{0}[\widehat{\Gamma}]$ such that $\mathcal{T}\vdash \widehat{\Gamma}\rightarrow (\widehat{\varphi}\leftrightarrow \psi)$; 
    \item if $\mathbf{emTT}_{\mathcal{T}}\vdash a\in A\,[\Gamma]$, then $\mathcal{T}\vdash \widehat{\Gamma}\rightarrow\exists u(\delta_a\wedge \eta_A)$;
    \item if $\mathbf{emTT}_{\mathcal{T}}\vdash a=b\in A\,[\Gamma]$, then $\mathcal{T}\vdash \widehat{\Gamma}\rightarrow\exists u(\delta_a\wedge \delta_b\wedge \eta_A)$.
\end{enumerate}
\end{theorem}
\begin{proof}This is a long but straightforward proof made simultaneously by induction on complexity of proof-trees in $\mathbf{emTT}_{\mathcal{T}}$. As an example, we show only one case relative to item 3., namely that of a small proposition obtained through a bounded universal quantifier with respect to a set. 

Assume that the judgement $(\forall x\in A)\varphi\,prop_s\,[\Gamma]$ is deduced in $\mathbf{emTT}_{\mathcal{T}}$ from the judgements $A\,set\,[\Gamma]$ and $\varphi\,prop_s\,[\Gamma,x\in A]$. Then, by inductive hypothesis, we know that $\mathcal{T}\vdash \widehat{\Gamma}\rightarrow \exists z\forall u(\eta_A\leftrightarrow u\in z)$ and that there exists a $\mathsf{K}_{0}[\widehat{[\Gamma,x\in A]}]$-formula $\psi$ such that $\mathcal{T}\vdash \widehat{[\Gamma,x\in A]}\rightarrow (\widehat{\varphi}\leftrightarrow \psi)$. The first one is equivalent to 
$$\mathcal{T}\vdash \widehat{\Gamma}\rightarrow \exists z\delta_{\lceil A\rceil}[z/u]$$
while from the second one we obtain
$$\mathcal{T}\vdash \widehat{\Gamma}\rightarrow \forall x( \eta_{A}[x/u]\rightarrow (\widehat{\varphi}\leftrightarrow \psi))$$
Assuming $\widehat{\Gamma}$, from these it follows in $\mathcal{T}$ that $\widehat{(\forall x\in A)\varphi}:\equiv \forall x(\eta_A[x/u]\rightarrow \widehat{\varphi})$ is equivalent to $\forall x(\eta_A[x/u]\rightarrow \psi)$ which is equivalent to $\exists z(\delta_{\lceil A\rceil}[z/u]\wedge \forall x\in z\,\psi)$. Since this last formula is a $\mathsf{K}_{0}[\widehat{\Gamma}]$-formula, we can conclude.

\end{proof}

\begin{cor}
Let $\varphi$ be a pre-proposition of $\mathbf{emTT}_{\mathcal{T}}$. If $\mathbf{emTT}_{\mathcal{T}}\vdash \mathsf{true}\in \varphi\,[\Gamma]$, then $\mathcal{T}\vdash \widehat{\Gamma}\rightarrow\widehat{\varphi}$.
\end{cor}
\begin{proof}
 From point 4.\ in Theorem \ref{emttzf}, it follows that if  $\mathbf{emTT}\vdash \mathsf{true}\in \varphi\,[\Gamma]$, then $\mathcal{T}\vdash \widehat{\Gamma}\rightarrow \exists u(\delta_\mathsf{true}\wedge \eta_\varphi)$ which means that $\mathcal{T}\vdash \widehat{\Gamma}\rightarrow\exists u(u=0\wedge \widehat{\varphi})$, Thus $\mathcal{T}\vdash \widehat{\Gamma}\rightarrow\widehat{\varphi}$.
\end{proof}

The theorems above show that the theories $\mathbf{emTT}_{\mathcal{T}}$ and $\mathcal{T}$ are equivalent. In fact, we have seen that:
\begin{enumerate}
    \item every formula of $\mathcal{T}$ is equivalent to one of the form $\widehat{\psi}$;
    \item every proposition of $\mathbf{emTT}_{\mathcal{T}}$ is equivalent to one of the form $\widetilde{\varphi}$; 
    \item if $\varphi$ is a theorem of $\mathcal{T}$, then $\widetilde{\varphi}$ is a theorem of $\mathbf{emTT}_{\mathcal{T}}$;
    \item if $\psi$ is a theorem of $\mathbf{emTT}_{\mathcal{T}}$, then $\widehat{\psi}$ is a theorem of $\mathcal{T}$.
\end{enumerate}
However, the relation between $\mathbf{emTT}_{\mathcal{T}}$ and $\mathcal{T}$ arising from the theorems above is much richer, as we will illustrate now.
\begin{enumerate}
    \item Let us denote with $\mathbf{Ctx}[\mathbf{emTT}_{\mathcal{T}}]$ the (meta)set of contexts of $\mathbf{emTT}_{\mathcal{T}}$, that is precontexts $\Gamma$ such that $\mathbf{emTT}_{\mathcal{T}}\vdash \Gamma\,context$ and with $\mathbf{Ctx}[\mathcal{T}]$ the (meta)set of pairs $([x_1,...,x_n],[\varphi_1,...,\varphi_n])$ where $[x_1,...,x_n]$ is a finite (possibly empty) list of variables and $\varphi_1,...,\varphi_n$ are formulas of $\mathcal{T}$ such that $\mathbf{free}(\varphi_i)\subseteq\{x_1,...,x_i\}$ for every $i=1,...,n$.

    An equivalente relation is defined on $\mathbf{Ctx}[\mathbf{emTT}_{\mathcal{T}}]$ as follows: $$[x_1\in A_1,...,x_n\in A_n]\equiv [y_1\in B_1,...,y_m\in B_m]$$ if and only if $n=m$, $x_i$ coincides with $y_i$ for every $i=1,...,n$ and $\mathbf{emTT}_{\mathcal{T}}\vdash A_i=B_i\,[x_1\in A_1,...,x_{i-1}\in A_{i-1}]$ for every $i=1,...,n$.  An equivalente relation is defined on $\mathbf{Ctx}[{\mathcal{T}}]$ as follows: $$([x_1,...,x_{n}],
    [\varphi_1,...,\varphi_n])\equiv' ([y_1,...,y_{m}],[\psi_1,...,\psi_m])$$ if and only if $n=m$, $[x_1,...,x_n]=[y_1,...,y_n]$ and $$\mathcal{T}\vdash \top\wedge \varphi_1\wedge...\wedge \varphi_{i-1}\rightarrow (\varphi_i\leftrightarrow \psi_i)$$ for every $i=1,...,n$. 

    A bijection between $\mathbf{Ctx}[\mathbf{emTT}_{\mathcal{T}}]/\equiv$ and $\mathbf{Ctx}[\mathcal{T}]/\equiv'$ is induced by the pair of functions
    $$\mathbf{Ctx}[\mathbf{emTT}_{\mathcal{T}}]\rightarrow \mathbf{Ctx}[\mathcal{T}]$$
    $$[x_1\in A_1,...,x_{n}\in A_n]\mapsto ([x_1,...,x_n],[\eta_{A_1}[x_1/u],...,\eta_{A_n}[x_n/u]])$$
    $$\mathbf{Ctx}[\mathcal{T}]\rightarrow \mathbf{Ctx}[\mathbf{emTT}_{\mathcal{T}}]$$
    $$([x_1,...,x_n],[\varphi_1,...,\varphi_n])\mapsto [x_1\in \{x_1|\,\widetilde{\varphi_1}\},...,x_n\in \{x_n|\,\widetilde{\varphi_n}\}]$$
    Notice that in the first case the conjunction of the list of formulas associated to a context $\Gamma$ of $\mathbf{emTT}_{\mathcal{T}}$ is exactly $\widehat{\Gamma}$.

    \item Let us fix a context $\Gamma\in \mathbf{Ctx}[\mathbf{emTT}_{\mathcal{T}}]$ and denote with $\mathbf{Col}_{\Gamma}$ the (meta)set of all pre-collections $A$ such that $\mathbf{emTT}_{\mathcal{T}}\vdash A\,col$ endowed with the equivalence relation defined by $A\equiv_\Gamma B$ if and only if $\mathbf{emTT}_{\mathcal{T}}\vdash A=B\,col\,[\Gamma]$. Moreover, denote with $\mathbf{DC}_{\widehat{\Gamma}}$ the (meta)set of definable classes (up to $\widehat{\Gamma}$) that is the (meta)set whose elements are formulas $\varphi$ such that $\mathbf{free}(\varphi)\subseteq \mathbf{free}(\widehat{\Gamma})\cup \{y\}$ with $y\notin \mathbf{free}(\Gamma)$; an equivalence relation on $\mathbf{DC}_{\widehat{\Gamma}}$ is given by $\varphi\equiv_{\widehat{\Gamma}}\psi$ if and only if $\mathcal{T}\vdash \widehat{\Gamma}\rightarrow (\varphi \leftrightarrow \psi)$.
    
    A bijection between $\mathbf{Col}_{\Gamma}/\equiv_\Gamma$ and $\mathbf{DC}_{\widehat{\Gamma}}/\equiv_{\widehat{\Gamma}}$ is induced by the following functions
    $$\mathbf{Col}_{\Gamma}\rightarrow \mathbf{DC}_{\widehat{\Gamma}}\qquad \mathbf{DC}_{\widehat{\Gamma}}\rightarrow \mathbf{Col}_{\Gamma} $$
    $$A\mapsto \eta_A[y/u]
\qquad\qquad \varphi\mapsto \{y|\,\widetilde{\varphi}\}$$
Moreover, this bijection restricts to a bijection between $\mathbf{Set}_{\Gamma}/\equiv_\Gamma$ and $\mathbf{DS}_{\widehat{\Gamma}}/\equiv_{\widehat{\Gamma}}$
where $\mathbf{Set}_{\Gamma}$ is the (meta)subset of $\mathbf{Col}_{\Gamma}$ consisting of those $A$ for which $\mathbf{emTT}_{\mathcal{T}}\vdash A\,set\,[\Gamma]$ and $\mathbf{DS}_{\widehat{\Gamma}}$ is the (meta)subset of $\mathbf{DC}_{\widehat{\Gamma}}$ consisting of those $\varphi$ for which $\mathcal{T}\vdash \exists z\forall y(y\in z\leftrightarrow \varphi)$ (where $z$ is fresh).
    \item Let us fix a context $\Gamma\in \mathbf{Ctx}[\mathbf{emTT}_{\mathcal{T}}]$ and denote with $\mathbf{Prop}_{\Gamma}$ the (meta)set of all pre-proposition $\varphi$ such that $\mathbf{emTT}_{\mathcal{T}}\vdash \varphi\,prop\,[\Gamma]$ endowed with the equivalence relation defined by $\varphi\equiv^p_\Gamma \psi$ if and only if $\mathbf{emTT}_{\mathcal{T}}\vdash \varphi = \psi \,prop\,[\Gamma]$. Moreover, denote with $\mathbf{Form}_{\widehat{\Gamma}}$ the (meta)set of formulas $\varphi$ of $\mathcal{T}$ such that $\mathbf{free}(\varphi)\subseteq \mathbf{free}(\widehat{\Gamma})$ endowed with the equivalence relation given by $\varphi\equiv^p_{\widehat{\Gamma}}\psi$ if and only if $\mathcal{T}\vdash \widehat{\Gamma}\rightarrow (\varphi \leftrightarrow \psi)$.
    
    A bijection between $\mathbf{Prop}_{\Gamma}/\equiv^p_{\Gamma}$ and $\mathbf{Form}_{\widehat{\Gamma}}/\equiv^p_{\widehat{\Gamma}}$ is induced by the functions
    $$\mathbf{Prop}_{\Gamma}\rightarrow \mathbf{Form}_{\widehat{\Gamma}}\qquad \mathbf{Form}_{\widehat{\Gamma}}\rightarrow \mathbf{Prop}_{\Gamma} $$
    $$\varphi\mapsto \widehat{\varphi}\qquad\qquad\qquad \qquad \psi\mapsto \widetilde{\psi}$$
    \item Let us now fix $\Gamma\in \mathbf{Ctx}[\mathbf{emTT}_{\mathcal{T}}]$ and $A\in \mathbf{Col}_{\Gamma}$. We denote with $\mathbf{Ext}(A)_{\Gamma}$ the (meta)set containing those pre-terms $a$ such that $\mathbf{emTT}_{\mathcal{T}}\vdash a\in A\,[\Gamma]$ endowed with the equivalence relation defined by $a\equiv^A_\Gamma b$ if and only if $\mathbf{emTT}_{\mathcal{T}}\vdash a=b\in A\,[\Gamma]$. Moreover, we denote with $\mathbf{DEl}(A)_{\widehat{\Gamma}}$ the (meta)set of formulas $\delta$ such that 
    \begin{enumerate}
        \item $\mathbf{free}(\delta)\subseteq \mathbf{free}(\Gamma)\cup \{y\}$ in which $y$ is a fresh variable;
        \item $\mathcal{T}\vdash \widehat{\Gamma}\rightarrow \exists ! y \,\delta$;
        \item $\mathcal{T}\vdash \widehat{\Gamma}\wedge \delta\rightarrow \eta_A[y/u]$.
    \end{enumerate}
    endowed with the equivalence relation defined by $\delta\equiv^A_{\widehat{\Gamma}}\delta'$ if and only if $\mathcal{T}\vdash \widehat{\Gamma}\rightarrow (\delta\leftrightarrow \delta')$. A bijection between $\mathbf{Ext}(A)_{\Gamma}/\equiv^{A}_{\Gamma}$ and $\mathbf{DEl}(A)_{\widehat{\Gamma}}/\equiv^{A}_{\widehat{\Gamma}}$ is determined by the following functions:
    $$\mathbf{Ext}(A)_{\Gamma}/\equiv^{A}_{\Gamma}\rightarrow \mathbf{DEl}(A)_{\widehat{\Gamma}}/\equiv^{A}_{\widehat{\Gamma}}\qquad \mathbf{DEl}(A)_{\widehat{\Gamma}}/\equiv^{A}_{\widehat{\Gamma}}\rightarrow \mathbf{Ext}(A)_{\Gamma}/\equiv^{A}_{\Gamma}$$
    $$\qquad \qquad a\mapsto \delta_a \qquad\qquad\qquad\qquad \qquad\qquad \delta\mapsto \bigcup\{y\in \lceil A\rceil|\,\widetilde{\delta}\}$$
    
\end{enumerate}
\section{Conclusions}
We have proven here that one can extend the extensional level of the Minimalist Foundation to the main intuitionistic and classical axiomatic set theories by adding rules and preserving the logical meaning. This provides a very strong notion of compatibility between these theories confirming the fact that the Minimalist Foundation is a suitable common ground for comparison between set-theoretical foundational theories and intuitionistic type theories.

\bibliography{bibliopsp}

\begin{thebibliography}{1}

\bibitem{czfnotes}
P.~{Aczel} and M.~{Rathjen}.
\newblock Notes on constructive set theory.
\newblock Available at http://www1.maths.leeds.ac.uk/$\sim$rathjen/book.pdf,
  2010.

\bibitem{BB85}
E.~{Bishop} and D.S. {Bridges}.
\newblock {\em Constructive analysis}.
\newblock Springer, 1985.

\bibitem{COC}
Thierry Coquand and G\'{e}rard Huet.
\newblock The calculus of constructions.
\newblock {\em Inform. and Comput.}, 76(2-3):95--120, 1988.

\bibitem{m09}
M.~E. {Maietti}.
\newblock A minimalist two-level foundation for constructive mathematics.
\newblock {\em Annals of Pure and Applied Logic}, 160(3):319--354, 2009.

\bibitem{mtt}
M.E. {Maietti} and G.~{Sambin}.
\newblock {Toward a minimalist foundation for constructive mathematics}.
\newblock In {L. Crosilla and P. Schuster}, editor, {\em From Sets and Types to
  Topology and Analysis: Practicable Foundations for Constructive Mathematics},
  number~48 in {Oxford Logic Guides}, pages 91--114. {Oxford University Press},
  2005.

\bibitem{ML84}
P.~{Martin-L\"{o}f}.
\newblock {\em Intuitionistic Type Theory. Notes by G. Sambin of a series of
  lectures given in Padua, June 1980}.
\newblock Bibliopolis, Naples, 1984.

\bibitem{NPS90}
B.~{Nordstr{\"o}m}, K.~{Petersson}, and J.~M. {Smith}.
\newblock {\em {Programming in Martin-L{\"o}f's Type Theory, an introduction}}.
\newblock Oxford University Press, 1990.

\end{thebibliography}
\end{document}